\documentclass[12pt]{article}

\usepackage{graphicx}
\usepackage{amsmath,amsthm,amssymb,color,dsfont}
\usepackage[noadjust,sort]{cite}

\makeatletter
\g@addto@macro\normalsize{%
  \setlength\abovedisplayskip{8pt plus 3pt minus 3pt}
  \setlength\belowdisplayskip{8pt plus 3pt minus 3pt}
  \setlength\abovedisplayshortskip{6pt plus 3pt minus 2pt}
  \setlength\belowdisplayshortskip{6pt plus 3pt minus 2pt}
}
\makeatother

\date{\today}

\normalsize

\numberwithin{equation}{section}

\def\({\bigl(}
\def\){\bigr)}

\newtheorem{thm}{Theorem}[section]

\newtheorem{lemma}[thm]{Lemma}
\newtheorem{conj}[thm]{Conjecture}

\theoremstyle{definition}

\def\abs#1{\lvert#1\rvert} \let\card=\abs
 
\def\norm#1{\lVert#1\rVert}

\def\dfrac#1#2{\lower0.15ex\hbox{\large$\textstyle\frac{#1}{#2}$}}
\def\({\bigl(}
\def\){\bigr)}
\def\st{\,:\,}

\def\fix{\triangleright}
\def\One{\mathds{1}}
\def\newDelta{\overset{\rule{0.5em}{0.4pt}}{\Delta}}

\let\eps=\varepsilon

\def\lambdavec{\boldsymbol{\lambda}}

\def\RG{\operatorname{RG}}
\def\D{\mathcal{D}}
\def\X{\boldsymbol{X}}
\def\Y{\boldsymbol{Y}}
\def\Z{\boldsymbol{Z}}
\def\x{\boldsymbol{x}}
\def\y{\boldsymbol{y}}
\def\calA{\mathcal{A}}
\def\calE{\mathcal{E}}

\def\calG{\mathcal{G}}

\def\thetavec{\boldsymbol{\theta}}
\def\xvec{\boldsymbol{x}}
\def\yvec{\boldsymbol{y}}
\def\zvec{\boldsymbol{z}}

\def\Om{\boldsymbol{\varOmega}}

\def\E{\operatorname{\mathbb{E}}}

\def\Var{\operatorname{Var}}

\def\Reals{{\mathbb{R}}}
\def\Complexes{{\mathbb{C}}}

\def\edge{\mathcal{E}}

\def\nicebreak{\vskip 0pt plus 50pt\penalty-300\vskip 0pt plus -50pt }

\begin{document}

\title{A tail bound for cumulant series for complex functions of independent random variables}
\author{Mikhail Isaev\\
School of Mathematics and Statistics\\
UNSW Sydney,
Sydney, Australia}
\date{}
\maketitle

\begin{abstract}
  We obtain explicit bounds on the truncation error of the cumulant series of a bounded complex function of a random vector with independent components.
  The error is given in terms of multidimensional differences corresponding to mixed partial derivatives.
  This extends the theory of the author with Brendan McKay and Rui-Ray Zhang (J. Combin.\ Th., Ser.\ B, 2025) from real functions to complex functions.
  We demonstrate some initial applications including a Berry--Esseen bound,
  an Edgeworth expansion for triangles in random graphs, and
  enumeration of regular graphs.
\end{abstract}

\section{Introduction}

The \textit{cumulants} $\kappa_j(W)$ of a bounded random variable $W$ can be defined
by the formal generating identity
\begin{equation}\label{cum:def}
   \E e^{tW} = 
    \sum_{j\ge 0} \frac{t^j \E W^j}{j!} = 
    \exp\biggl(\sum_{r\ge 1} \frac{\kappa_r(W) z^j}{r!}\biggr);
\end{equation}
see Speed~\cite{Speed} for the basic properties of cumulants.
Note that the sum over $j$ converges for all $t$, but the sum over $r$ might not.
The first few cumulants are given by the formulas 
\begin{align*}
    \kappa_1(W) &= \E W, \qquad \kappa_2(W) = \E(W-\E W)^2,  \qquad \kappa_3(W) = \E(W-\E W)^3,\\
    \kappa_4(W) &= \E(W-\E W)^4 - 3 \left(\E(W-\E W)^2\right)^2.
\end{align*}

In this paper we are concerned with the accuracy of the approximation of $\E e^{W}$
by  the sum over $r$ in \eqref{realThm}  truncated to
\[
    \exp\biggl(\sum_{r=1}^m \frac{\kappa_r(W) z^j}{r!}\biggr),
\]
when  $W$ is a complex function of several independent random variables.
The case of $m=2$ for real functions was treated by Catoni~\cite{Catoni}.   
This was extended to complex functions in the more general setting of martingales
by McKay and the author in \cite{mother}. For $m> 2$ and real functions, the tail bound for the cumulant series was obtained by McKay, Zhang and the author in  \cite{eulerian}, which we present below as Theorem \ref{realThm}.

Let $\Om=\varOmega_1\times \cdots\times \varOmega_n$ be an $n$-dimensional
domain and let $[n]$ denote the set $\{1, \ldots, n\}$. 
Suppose $\X=(X_1,\ldots, X_n)$ is a random vector with independent components, taking values in~$\Om$. For vectors $\xvec=(x_1,\ldots,x_n),\yvec=(y_1,\ldots,y_n)\in\Om$ and $V\subseteq[n]$, define $\xvec\fix_V\yvec = (u_1,\ldots,u_n)$
where, $u_j=x_j$ if $j\notin V$ and $u_j=y_j$ if $j\in V$.
For a bounded measurable function $f:\Om\to\Complexes$ and $V\subseteq[n]$, define
\begin{equation}\label{def:Delta}
   \Delta_V(f) = \sup_{\xvec,\yvec\in\Om}\;
   \biggl| \sum_{W\subseteq V} (-1)^{\card{W}} f(\xvec\fix_W\yvec)
   \biggr|.
\end{equation}

\begin{thm}[Isaev, McKay, Zhang \cite{eulerian}]\label{realThm}
	Let $m > 0$ be an integer and $\alpha \geq 0$. 
	Suppose $f:\Om\to\Reals$ is bounded and measurable,
    such that for all $v\in[m]$,
    \[
       \max_{j \in [n]}\,\sum_{V\in \binom{[n]}{v} \st j \in V } \!\Delta_V(f)
       \le \alpha.
    \]
	Then, we have
   \[
  	 \E { e^{f(\X)} }  =  
   (1+\delta)^n  \exp\biggl(\sum_{r=1}^m \frac{\kappa_r (f(\X))}{r!} \biggr),
  \]
   where $\delta> -1$ satisfies $\abs\delta \leq  e^{(100 \alpha)^{m+1}}-1$. 
   Furthermore,
   for any $r\in [m]$,
   \[
            \abs{\kappa_r (f(\X))}
            \leq  n \frac{(r-1)!}{50r}(80
             \alpha)^{r}.  
   \]
\end{thm}

Theorem~\ref{realThm} is not true in the case of complex-valued $f$
(see \cite[Remark 2.5]{mother} for a counterexample) but in this paper we will fill that gap under additional assumptions on $\Delta_V$ for larger sets. Intuitively, this is required to control the cancellations due to possible oscillations that might depend on many variables. 

For $V\subseteq [n]$, define
\begin{equation}\label{def:newDelta}
   \newDelta_V(f,\X) = \sup_{\yvec\in\Om}\;
   \biggl| \E\biggl(\sum_{W\subseteq V} (-1)^{\card{W}} f(\X\fix_W\yvec)
   \biggr)\biggr|.
\end{equation}
Clearly $\newDelta_V(f,\X)\le\Delta_V(f)$.
For example, for $n=3$,
\begin{align*}
    \newDelta_{\{1\}} &= \sup_{x_1}\,
      \bigl| \E f(X_1,X_2,X_3) - \E f(x_1,X_2,X_2) \bigr| \\
     \newDelta_{\{1,2\}} &= \sup_{x_1,x_2}
      \bigl| \E f(X_1,X_2,X_3) - \E f(x_1,X_2,X_3) 
         - \E f(X_1,x_2,X_3) + \E f(x_1,x_2,X_3)\bigr|.
\end{align*}
Further define, for $k\in[n]$, the quantity
\begin{equation}\label{Sdef}
   S_k(f,\X) = \max_{j \in [n]}\,\sum_{V\in \binom{[n]}{k} \st j \in V } \!\newDelta_V(f,\X).
\end{equation}

\begin{thm}\label{complexThm}
	Let $m > 0$ be an integer and $\alpha \ge 0$. 
	Suppose $f:\Om\to\Complexes$ is bounded and measurable such that
    \[
       \sum_{k=t}^n e^{3\alpha k/2} 2^t\binom{k-1}{t-1}S_k(f,\X) \le \alpha \leq \dfrac{1}{100} \qquad \text{for  all $t \in [m]$.}
    \]
	Then, we have
   \[
  	 \E { e^{f(\X)} }  =  
   (1+\delta)^n  \exp\biggl(\sum_{r=1}^m \frac{\kappa_r (f(\X))}{r!} \biggr),
  \]
   where $\delta$ satisfies $\abs\delta \leq  e^{(100 \alpha)^{m+1}}-1$. 
   Furthermore, for $2\le r\le m$,
   \[
            \abs{\kappa_r (f(\X))}
            \leq  n \frac{(r-1)!}{50r}(80
             \alpha)^{r}.  
   \]
\end{thm}

The proof of Theorem~\ref{complexThm} will be deferred to Section \ref{S:decompositions}.
Section~\ref{S:applications} will demonstrate some initial applications,
including a central limit theorem with the convergence rate estimate of the Berry--Esseen type
and an Edgeworth expansion for triangle counts in random graphs.
We will also describe how Theorem~\ref{complexThm} can be used to obtain high-precision results in combinatorial enumeration, using the number of regular
graphs as an example.

%%%%%%%%%%%%%%%%%%%%%%
%%%%%%%%%%%%%%%%%%%%%%
%%%%%%%%%%%%%%%%%%%%%%
%%%%%%%%%%%%%%%%%%%%%%
%%%%%%%%%%%%%%%%%%%%%%
%%%%%%%%%%%%%%%%%%%%%%
\section{Some applications}\label{S:applications}
In this section,  
we illustrate the power of Theorem \ref{complexThm} with some  
straightforward applications. 
There are many interesting generalisations and  related questions along these directions, but they deserve careful consideration in subsequent papers.

Recall the definitions of $\Delta_V(f)$ and $\newDelta_V(f,\X)$  from  \eqref{def:Delta} and  \eqref{def:newDelta},  respectively. Note that they are  subadditive:
\begin{equation*} 
    \Delta_V(f+g)\leq \Delta_V(f)+\Delta_V(g), \qquad 
     \newDelta_V(f+g,\X)\leq \newDelta_V(f,\X)
     +\newDelta_V(g,\X).
\end{equation*}
Consequently, $S_k(f,\X)$ defined in \eqref{Sdef} is also subadditive.
We will use  the following bound $\Delta_V(f)$   obtained from the derivatives of~$f$.

\begin{lemma}\label{derivbnd}
Suppose $\Om=[a_1,b_1]\times\cdots\times [a_n,b_n]$ is a product of finite real
intervals, and that
$f:\Om\to\Complexes$ is analytic in~$\Om$.
Then, for all $V\subseteq [n]$,
\[
   \newDelta_V(f,\X)\leq \Delta_V(f) \le \prod_{j\in V}\, (b_j-a_j) \cdot \max_{\xvec\in\Om}
   \;\biggl| \frac{\partial^{\card V} f(\xvec)}{\prod_{j\in V} \partial x_i}
   \biggr|,
\]
where $\xvec=(x_1,\ldots,x_n)$.
\end{lemma}
\begin{proof}
 The representative case $V=\{1,2\}$ follows from
 \begin{align*}
    \Delta_{\{1,2\}} & =\max_{y_1,y'_1,y_2,y_2'}\,\bigl| f(y_1,y_2,\ldots) - f(y'_1,y_2,\ldots)
    - f(y_1,y'_2,\ldots) + f(y'_1,y'_2,\ldots)\bigr| \\
    &= \max_{y_1,y'_1,y_2,y_2'}\,\biggl|\int_{[y_2,y'_2]}\int_{[y_1,y'_1]}
       \!\! \frac{\partial^2 f(x_1,x_2,\ldots)}
            {\partial x_1\partial x_2}  \,dx_1dx_2 \biggr|\\
    &\le (b_1-a_1)(b_2-a_1)\max_{x_1,x_2}
        \;\biggl| \frac{\partial^2 f(x_1,x_2,\ldots)}
            {\partial x_1\partial x_2} \biggr|. \qedhere
 \end{align*}
\end{proof}

More generally, if $f(\xvec)=g(T\xvec)$ for a linear operator $T$, then bounds on $\Delta_V(f)$ in terms of the derivatives of~$g$ are given in~\cite[Lemma~3.5]{eulerian} for real functions but are equally valid for complex functions. For certain applications where $\X$ is a Gaussian random variable,  one can first truncate $\X$ to a multidimensional cube containing most of the measure to ensure that the quantities $\Delta_V(f)$ are sufficiently small and then use \cite[Lemma 3.4]{eulerian} to show that the cumulants remain essentially the same.

\subsection{Berry--Esseen-type   bound}

The classical Berry--Esseen theorem provides a uniform bound on the difference between 
the normal  distribution function
$$
 \Phi(x) = \frac{1}{\sqrt{2\pi}} \int\limits_{-\infty}\limits^{x}	e^{-\frac{t^2}{2}} dt
$$
and the cumulative distribution function  of 
$
 \frac{1}{\sqrt{n}}\sum_{k=1}^{n} X_i,
$
where $X_i$ are i.i.d. real random variables with mean $0$, variance $1$ and $\E |X_i|^3 = \rho$. It states
\begin{equation*} 
	\sup_{x\in \Reals}\left|\Pr\left( \frac{\sum_{i=1}^{n} X_i}{ \sqrt{n}} < x\right) - \Phi(x)\right| \leq \frac{C\rho}{ n^{1/2}},
\end{equation*}
where $C$ is a universal constant. One  way  to prove the Berry-Essen theorem  is by Feller's smoothing lemma (see \cite[Section XVI.3]{Feller1971}):
\begin{equation}\label{Feller}
		\pi |F(x) - \Phi(x)| \leq \int\limits_{-T}\limits^T \bigl|\varphi(t) - e^{-\frac{t^2}{2}}\bigr| \frac{dt}{|t|} + \frac{24}{ \sqrt{2\pi} T} 
	 \ \ 	\text{ for all $x$ and $T>0$.}
\end{equation}
Here $F(x)$ and  $\varphi(t)$  are the cumulative distribution function and characteristic function of any real random variable.

Theorem \ref{complexThm}  is  a convenient tool for estimating $\varphi(t)$ for functions of independent random variables when $t$ is not too large. Thus, combining it together with \eqref{Feller}, one can derive the following Berry--Essen-type bound.

\begin{thm}\label{T:BE}
	Suppose $f:\Om\to\Reals$ is bounded and measurable with $\sigma^2 = \Var f(\X)$
     and 
     \[
     S_k(f,\X) \leq \dfrac{a}{k} e^{- k w} \qquad \text{for some $a>0$, $w\in (0,1)$ and all $k \in [n]$.}
    \]
     Then, uniformly  for all $|t| \leq  \dfrac{w^2 \sigma }{800\,a}$, 
   \[
  	 \varphi(t) := \E e^{it \frac{f(\X)-\E f(\X)}{\sigma}} = \exp\left(- \dfrac{t^2}{2} + O\left(\dfrac{n a^3 t^3}{w^3\sigma^{3}}\right)\right)  
  \]
  and
  \[
    \max_{x\in \Reals}\left | \Pr (f(\X) <x) - \Phi\Big(\dfrac{x- \E f(\X)}{\sigma}\Big) \right| = O\left(\dfrac{a}{w^2 \sigma}+ \dfrac{n a^3 }{ w^3\sigma^3 } \right).
  \]
\end{thm}

\begin{proof}
% We can assume $c=1$ without loss of the generality since the claims are invariant with respect to scaling.
 We apply Theorem \ref{complexThm} to the function $g(\X) :=  \dfrac{it}{\sigma}(f (\X) - \E f(\X))$. Let $\alpha = \dfrac{8at}{w \sigma}$ 
 and $|t| \leq \dfrac{w^2 \sigma}{800a}$.
Since $w\leq 1$ we get $\alpha \leq \frac{w}{100} \leq \frac{1}{100}$.
 By the assumptions, we find that 
 \[
    S_k(g,\X) \leq \dfrac{at}{k\sigma} e^{-k w}.
 \]
 Thus,   we can estimate 
 \begin{align*}
    \sum_{k=1}^m e^{\frac{3}{2}\alpha k} k S_k(g)
    &\leq \dfrac{at}{\sigma} \sum_{k=1}^m e^{3\alpha k/2 - k w}
    \leq \dfrac{at}{\sigma} 
    \sum_{k=1}^\infty e^{ - 0.98k w}
    \\
    &= \dfrac{at}{\sigma}
    \left(1-e^{-0.98w}\right)^{-1}  \leq \dfrac{2at}{w\sigma}=
    \alpha/4.
 \end{align*}
 This implies the assumptions of Theorem \ref{complexThm} with $m=2$. Applying Theorem \ref{complexThm}, we get that 
 the required bound for $\varphi(t)$.

 To estimate $\Pr(f(\X)<x)$ we can assume that $\sigma w/a  = \omega(n^{1/3})$, since the claimed bound is trivial otherwise as the LHS is always at most 1. 
  Let 
  \[ 
  T = \min\left\{\dfrac{w^2\sigma }{a}, \dfrac{w^3\sigma^3 }{a^3 n} \right\}\quad \text{and} \quad  
   T' = \min\left\{T,\dfrac{w\sigma }{a n^{1/3}}\right\}. 
   \]
   Using our estimate for $\varphi(t)$, we find that 
   \[
        \int\limits_{-T'}\limits^{T'} |\varphi(t) - e^{-\frac{t^2}{2}}|
        \frac{dt}{|t|} =
        \int\limits_{-T'}\limits^{T'} O\left(\dfrac{na^3t^3}{w^3 \sigma^3} \right) e^{-\frac{t^2}{2}}
        \frac{dt}{|t|} = O\left(\dfrac{n a^3}{w^3 \sigma^3}\right).
   \]
   The contribution of the middle range is negligible: 
   \[
    \int\limits_{T'\leq |t|\leq T} \left|\varphi(t) - e^{-\frac{t^2}{2}}\right|
        \frac{dt}{|t|} \leq 
           \int\limits_{T'\leq |t|\leq T}  2e^{-(1+o(1)) \frac{t^2}{2}}\frac{dt}{|t|} = e^{-\Omega(T'^2)} = O\left(\dfrac{a}{w^2 \sigma}+ \dfrac{n a^3 }{ w^3\sigma^3 } \right).
   \]
   Using \eqref{Feller}, we have completed the proof.
\end{proof}

For the case when $f(\X) = \sum_{i=1}^{n} X_i$ and i.i.d.\ bounded $X_i$, we get that $S_1(f,\X) = O(1)$ and $S_k(f,\X) = 0$ for $k\geq 2$. Applying Theorem \ref{T:BE}   gives the optimal rate of convergence $O(n^{-1/2})$. Similarly, one can derive optimal Berry-Esseen bounds for $U$-statistics (possibly with some missing terms) or some mixtures of those.   In particular, Theorem~\ref{T:BE}  immediately implies the asymptotic normality of small subgraph counts in the binomial random graph   $\calG(n,p)$  (for fixed $p\in (0,1)$) or similar models with independent adjacencies.  We omit the details here, as
Ruciński \cite{Rucinski1988} already established 
a full characterisation when small subgraph counts obey a central limit theorem for the general case $p=p(n)$ (which can be a decreasing function).

For many applications, stronger methods for establishing asymptotic normality are available -- such as the martingale central limit theorem or the Stein method -- which do not impose the same structural restrictions on 
$f$ as Theorems \ref{complexThm} and \ref{T:BE}. However, these approaches rarely yield optimal rates of convergence and are therefore unsuitable for obtaining sharper approximations, as noted in \cite{BGZ1997}.

% In particular, Theorem~\ref{T:BE} is useful to study distribution  of small subgraph counts in   $\calG(n,p)$. 

 Furthermore, Theorem \ref{complexThm}   
 can give more precise estimates for  $\varphi(t)$  when $t$ is not too large. For bigger values of $t$,  a crude upper bound for $|\varphi(t)|$ is sufficient because $e^{-t^2/2}$ is already negligible. 
 Depending on the range, where one can estimate $|\varphi(t)|$, this can lead to Edgeworth-type approximations for the distribution of $f(\X)$. 
In Section \ref{S:triangles}, we illustrate this using triangle counts in   $\calG(n,p)$.

A prominent idea that helps with the large $t$ is known as \emph{decoupling}.
In particular, this idea was crucial in \cite{AM2023, Berkowitz2018, GK2016, SS2022} for establishing local limit theorems for small subgraph counts in $\calG(n,p)$.  Here is a brief description.
Suppose we can split $\X = (\Y,\Z)$, where $\Y,\Z$ are independent. Then, for any event $\calA$, 
\begin{align*}
    |\E e^{\frac{it}{\sigma}f(\Y,\Z)}|^2 &=
    \E e^{\frac{it}{\sigma}f(\Y,\Z)}  \E e^{-\frac{it}{\sigma}f(\Y,\Z')}
    =
    \E e^{\frac{it}{\sigma}(f(\Y,\Z) - f(\Y,\Z'))}  
    \\  &\leq 2\Pr(\Z \notin \calA) + \E \left(e^{\frac{it}{\sigma}(f(\Y,\Z) - f(\Y,\Z'))} \mid \Z,\Z' \in \calA\right),
\end{align*}
where $\Z'$ is an independent copy of $\Z$.
The main point is that for typical $\zvec$ and $\zvec'$ there can be a lot of cancellations in $g_{\zvec,\zvec'}(\Y)=f(\Y,\zvec) - f(\Y,\zvec')$ which, in particular, reduce the quantities  $\newDelta_V$ or even sometimes produce a sum of independent random variables after several rounds of decoupling. Then, one can achieve the desired bound for $|\varphi(t)|$ by analysing $\E e^{\frac{it}{\sigma}g_{\zvec,\zvec'}(\Y)}$.

\subsection{Edgeworth expansion for triangle counts}\label{S:triangles}

Let $T= T(n,p)$ be the number of triangles in the binomial random graph $\calG(n,p)$, where every pair appear as an edge with probability $p$ independently of each other. 
In this section, we focus on  estimating its probability mass function $\Pr(T = k)$.

For any fixed $c\in {0,1}$ and 
$n^{-1}\ll p \leq c$,  the normal  approximation holds:
\[
    \sigma \Pr(T = k)  =  \phi\left(\frac{k-\mu}{\sigma}\right) + o(1),
\]
where 
$\phi(x)=\dfrac{1}{\sqrt{2\pi}}e^{-x^2/2}$ is the normal density function,
\begin{align*}
    \mu&:= \E T = \binom{n}{3}p^3,\\
    \sigma^2 
    &:= \Var T =  12\binom{n}{4} (p^5 - p^6)
    + \binom{n}{3}p^3 (p^3 - p^6).
\end{align*}
The case of fixed $p$ is due to Gilmer and Kopparty \cite{GK2016}. R\"ollin and Ross \cite{RR2015}   showed the local limit theorem for   $n^{-1}\ll p \leq O(n^{-1/2})$. The remaining gap was recently filled by  Ara\'ujo and  Mattos~\cite{AM2023}.
Furthermore, Berkowitz \cite{Berkowitz2018}  showed that it is true for clique counts in the case of fixed $p$.
In fact, for fixed $p$,   the local limit theorem applies to all small connected subgraph counts, as established by Sah and Sawhney in \cite{SS2022}.

Using Theorem \ref{complexThm}, we can obtain more accurate approximations for   $\Pr(T =k)$ based on the cumulant series for its characteristic function. For   $r\geq 2$, let 
\[
\lambda_r := \kappa_r(T)/\sigma^{r}.
\]
For $m\geq 0$, consider the truncated  Edgeworth  series 
\begin{equation*} 
    \edge_{m}(x) =  \sum_{s=0}^{m} P_{s,\lambdavec}(-D) [\phi](x), 
\end{equation*}
where $D[\phi] = \frac{d}{dx}\phi$ and 
the polynomials $P_{s,\lambdavec}$ are defined by
\begin{equation}\label{def:poly}
    P_{s,\lambdavec}(y) = \sum_{\pi} 
    \prod_{j} \frac{1}{k_j!} \left(\frac{\lambda_{j+2}\,y^{j+2}}{(j+2)! } \right)^{k_j}.
\end{equation}
Here, the summation is over all the integer partitions $\pi$ of  $[s]$ 
and $k_j$ denotes  the number of parts of size $j$ so, in particular, we have that  $\sum_j j k_j = s$.
For example, here are the first few terms:
\begin{align*}
     \edge_{4}(x) = 
  \phi(x)
  &- \left( \frac{1}{6} \lambda_3 \, \phi^{(3)}(x)\right)   
  +\left( \frac{1}{24} \lambda_4 \, \phi^{(4)}(x) + \frac{1}{72} \lambda_3^2 \, \phi^{(6)}(x) \right) \\
& -   \left( \frac{1}{120} \lambda_5 \, \phi^{(5)}(x) + \frac{1}{144} \lambda_3 \lambda_4 \, \phi^{(7)}(x) + \frac{1}{1296} \lambda_3^3 \, \phi^{(9)}(x) \right) \\
& +   \Bigg( \frac{1}{720} \lambda_6 \, \phi^{(6)}(x) +
 \left( \frac{1}{1152} \lambda_4^2 + \frac{1}{720} \lambda_3 \lambda_5 \right) \phi^{(8)}(x) 
 \\ &\hspace{5cm}+ \frac{1}{1728} \lambda_3^2 \lambda_4 \, \phi^{(10)}(x) 
+ \frac{1}{31104} \lambda_3^4 \, \phi^{(12)}(x) \Bigg).
\end{align*}
 % Note that scaled cumulants $\lambda_r$ are invariant under linear transformations of $X$, so the truncated Edgeworth series $\edge_{r,X}(x)$ are also invariant.
\begin{thm}\label{T:triangles}
  For any  fixed $p\in (0,1)$ and   $m\geq 0$, we have 
\[
\sup_{k}\left|\sigma \Pr(T=k)- \calE_{m}\left(\frac{k-\mu}{\sigma}\right)\right|= O(n^{-m-1}).
\] 
\end{thm}
\begin{proof}
We approach the distribution of $T$ using the Fourier inverse formula: for any $k$,
\begin{equation}\label{eq:start}
\Pr\left(T = k\right)=\frac{1}{\sqrt{2\pi} \sigma} \int_{-\pi \sigma}^{\pi \sigma} e^{-it \frac{k-\mu}{\sigma}} \varphi(t) dt,
\end{equation}
where 
  $
    \varphi(t) = \E e^{it \frac{T-\mu}{\sigma}}.
$
Berkowitz showed in \cite[Corollary 4 and Lemma 18]{Berkowitz2018}  that 
\begin{equation}\label{eq:Berk}
   \max_{n^{2/3}\leq |t|\leq \pi \sigma} | \varphi(t)| \leq e^{-\Omega(n^{\eps})}.
\end{equation}
Thus, we only need to estimate $\varphi(t)$ for $|t|\leq n^{2/3}$. We will apply Theorem \ref{complexThm} for 
$f(\X) = it \frac{T-\mu}{\sigma}$, where $\X \in \{0,1\}^{\binom{n}{2}}$ is the vector of edges of $\calG(n,p)$. Applying Lemma~\ref{derivbnd}, we find that, for any edge $e$, 
\[
    \Delta_{\{e\}}(f) \leq \dfrac{t}{\sigma} n.
\]
This implies that  $S_1  = O(t/n)$.
Next, for any two edges $e,e'$,  we find that 
\[
    \Delta_{\{e,e'\}}(f) \leq 
    \begin{cases}
        t/\sigma, &\text{if $e$ and $e'$ are adjacent}\\
        0, &\text{otherwise}.
    \end{cases}
\]
This implies that $S_2 = O(t/n)$.  Similarly,
observing $\Delta_{\{e,e', e''\}}(f) =0$ for any  three edges $e,e',e''$ unless they form a triangle, we derive that $S_3 = O(t/n)$. Finally, we have $S_k =0$ for all $k>3$. Then, one can easily check that assumptions of Theorem \ref{complexThm} holds with $\alpha = O(t/n) = O(n^{-1/3})$ and any $m$. Applying Theorem \ref{complexThm}, we find that the cumulant series actually converges  for   $|t|\leq n^{2/3}$:
\begin{align*}
    \varphi(t) =  \exp \left(\sum_{r=1}^\infty
    \frac{\kappa_r(f(\X))}{r!} \right) 
    =  
    \exp \left(-\frac{t^2}{2}+ \sum_{r=3}^\infty \frac{\lambda_r (it)^r}{r!}  \right),  
\end{align*}
and $\dfrac{|\lambda_r|}{r!} \leq  \dfrac{C^r}{r^2 n^{r-2}}$ for some sufficiently large constant $C$. In fact, the radius of convergence is at least $\Omega(n)$, but we will not need this observation. Now it is a routine to expand the formula above and collect the terms according to the powers of $n$. 
The terms with $\lambda_r (it)^r$ for $r\geq 3m+9$ can be bounded by $O(n^{-m-1})$ since $|t|\leq n^{2/3}$.
After doing that, we conclude  that 
\[
\varphi(t) = e^{-t^2/2} \left(\sum_{s=0}^{m} P_{s,\lambdavec}(it) + O(n^{-m-1}) Q_m(t)\right)
\]
where polynomials 
$P_{s,\lambdavec}$ are defined in \eqref{def:poly} 
and $Q_m(t)$ is some polynomial of no importance.  Using \eqref{eq:Berk},
we get that, for any $x \in \Reals$,   
\begin{align*}
\frac{1}{\sqrt{2\pi}}   &\int_{-\pi \sigma}^{\pi \sigma} e^{-it x} \varphi(t) dt 
= e^{-\Omega(n^{\eps})} + 
\frac{1}{\sqrt{2\pi}}   \int_{-n^{2/3}}^{n^{2/3}} e^{-it x} \varphi(t) dt
\\
&= e^{-\Omega(n^{\eps})} + 
\frac{1}{\sqrt{2\pi}}   \int_{-n^{2/3}}^{n^{2/3}} e^{-it x -t^2/2} \left(\sum_{s=0}^{m} P_{s,\lambdavec}(it) + O(n^{-m-1}) Q_m(t)\right)dt 
\\
&= O(n^{-m-1}) + 
\frac{1}{\sqrt{2\pi}}   \int_{-\infty}^{\infty} e^{-it x -t^2/2} \left(\sum_{s=0}^{m} P_{s,\lambdavec}(it)\right) dt 
  \\&= O(n^{-m-1})+
  \calE_{m}\left(x\right).
\end{align*}
Taking $x= \frac{k - \mu}{\sigma}$ and using \eqref{eq:start}, the result follows.
\end{proof}

In a subsequent paper, we plan to extend Theorem \ref{T:triangles} to the case when $p=o(1)$ and provide more precise local estimates in the tail of the distribution.

%%%%%%%%%%%%%%%%%%%%%%
%%%%%%%%%%%%%%%%%%%%%%
%%%%%%%%%%%%%%%%%%%%%%
%%%%%%%%%%%%%%%%%%%%%%

\subsection{Enumeration of regular graphs}

Many combinatorial counts can be expressed in terms of high-dimensional integrals,  resulting from Fourier inversion applied to a multivariable generating function. By employing the saddle-point method, one can derive approximate formulas through estimates of expressions of the form $\E e^{f(\X)}$, where $\X$ follows a gaussian distribution truncated to a box $B$. A systematic framework for this approach, together with numerous examples spanning over 30 prior papers, can be found in \cite{mother}, which is based on the use of complex martingales.  For all of these examples, Theorem \ref{complexThm}  provides much more precise formulas with very little additional work (mainly estimating quantities $\Delta_V$). 
Moreover, Theorem \ref{complexThm} significantly extends the range of applicability of analytical methods. As an illustration, we consider the problem of counting regular graphs.

For even $dn$, let $\RG(n,d)$ denote the number of labelled $d$-regular graphs on $n$-vertices.
 The main term, for the full range $1\leq d\leq n-2$ is
 \begin{equation*} 
  \RG (n,d) \sim \sqrt2\, e^{1/4}
     \bigg(\lambda^\lambda (1{-}\lambda)^{1{-}\lambda}\bigg)^{\binom n2}
      \binom{n-1}{d}^{\!n}, \qquad \text{where } \lambda = \dfrac{d}{n-1}.
\end{equation*}
The key papers for this enumeration problem are the following.
\begin{itemize}
\item For $d=o(n^{1/2})$, McKay and Wormald ~\cite{MWsparse}  used the switching method to estimate the probability that no multiple edges or loops occur in the pairing model. 
\item For $\min\{d,n-d-1\} \geq cn/\log n$ where
 $c>\frac23$, McKay and Wormald \cite{MWreg}  employed  analytical methods to estimate complex integrals.
\item 
For a large range that overlaps
both the sparse and dense regimes, 
Liebenau and Wormald  \cite{Liebenau} used contracting properties of a recurrence   relating edge probabilities and graph counts.
\end{itemize}

In the follow-up paper \cite{regular}, we obtain a more refined estimate by taking further terms of the cumulant series into account. In fact, \cite{regular}  establishes a much more general result that allows for enumerating graphs with given degrees and large forbidden structures.

\begin{thm}[Isaev, McKay \cite{regular}]\label{regularthm}
Let $\varLambda = \lambda(1-\lambda)$.
There are
polynomials $p_j(x)$ for $j\geq 1$, with $p_j$ having degree
$j$ for each~$j$, such that
\begin{equation*} 
   \RG (n,d) = \sqrt2 \,
     \Big(\lambda^\lambda (1{-}\lambda)^{1{-}\lambda}\Big)^{\binom n2}
      \binom{n-1}{d}^{\!n}
   \exp\biggl( \,
        \sum_{j=1}^m \frac{p_j(\varLambda)}{\varLambda^j n^{j-1}} 
          + O(\varLambda^{-m-1}n^{-m}) \biggr)
\end{equation*}
for any fixed~$m\geq 1$,
provided $\min\{d,n-d-1\}\geq n^\sigma$ for some fixed $\sigma>0$.
\end{thm}

\begin{proof}[Proof ideas for Theorem \ref{regularthm}]
Here, we briefly sketch the main ideas and explain  
 how Theorem \ref{complexThm} enters the argument. First, observe that 
\begin{equation}\label{reg:generating}
    \RG(n,d) = [z_1^d\cdots z_n^d] \prod_{j<k} \,(1+z_jz_k),
\end{equation}
 where $[\cdot]$ denotes coefficient extraction. 
 We can find it using the Cauchy theorem.
\begin{equation}\label{RGint}
   \RG(n,d) = \frac{1}{(2\pi i)^n}
     \oint\!\cdots\!\oint \,
     \frac{\prod_{j<k} (1+z_jz_k)}
          {\prod_{j=1}^n z_j^{d+1}}\, dz_1\cdots dz_n,
\end{equation}
where the contours enclose the origin once anticlockwise.
 Taking  the circular contours $z_j = \sqrt{\dfrac{\lambda}{1-\lambda}} e^{i\theta_j}$ for all $j\in [n]$ gives
\[
    \RG(n,d) = (2\pi)^{-n} \left(\lambda^{\lambda}(1-\lambda)^{1-\lambda}\right)^{-\binom{n}{2}}   J ,
\]
where 
\begin{equation*}
J = \int_{-\pi}^\pi\!\cdots\!\int_{-\pi}^\pi
                  F(\thetavec)\,d\thetavec, \qquad 
  F(\thetavec) = \frac{\prod_{j<k}\,
                    \(1+\lambda (e^{i(\theta_j+\theta_k)}-1)\)}
               {e^{id \sum_{j=1}^n \theta_j}}. 
\end{equation*}
and $\thetavec = (\theta_1,\ldots,\theta_n)\in [-\pi,\pi]^n$, 
The integrand $F(\thetavec)$ takes its maximum absolute
value of~1 at two equivalent points 
$(0,\ldots,0)$  and  $(\pi,\ldots,\pi)$.   For a small $\eps>0$, define the region 
\begin{align*} 
   B&= \left\{\thetavec \in [-\pi,\pi]^n \st \max_{j\in [n]}|\phi_j|\leq \frac{n^{\eps}}{\sqrt{\varLambda n}}\right\},\\
    \phi_j &= \theta_j + \dfrac{1-\beta}{\beta n} \sum_{j\in [n]} \theta_j, \qquad \beta = \sqrt{\dfrac{n-2}{2(n-1)}}.
\end{align*}
One of the steps in \cite{regular} is to show that the contribution of all points outside of $B$ (and its equivalent region around $(\pi,\ldots,\pi)$) is negligible.
\[
    J = (2+e^{-\omega(\log n)}) \int_{B} F(\thetavec)d\thetavec.
\]
  For  $\thetavec \in B$, we have 
  \begin{equation*} 
    \theta_j = \phi_j - \dfrac{1-\beta}{n} \sum_{j\in[n]}\phi_j = O\left(\frac{n^{\eps}}{\sqrt{\varLambda n}}\right).
  \end{equation*}
Then, using Taylor's theorem, we can write 
\begin{align*}
    F(\thetavec) = \exp\biggl( - id \sum_{j\in [n]}\theta_j + \sum_{\ell\geq 1}\sum_{j<k} c_\ell (\theta_j+\theta_k)^\ell\biggr), 
\end{align*}
where $c_\ell$ are the coefficients of expansion 
\[
    \log \bigl(1+\lambda (e^{ix}-1)\bigr) = \sum_{\ell \geq 1}  c_{\ell} x^{\ell}
    =  i\lambda x - \dfrac{\varLambda}{2} x^2 - \dfrac{i\varLambda(1-2\lambda)}{6} x^3 + \cdots
\]
 Observe that 
 \begin{align*}
      \sum_{j<k} c_1 (\theta_j+\theta_k) &=   id \sum_{j\in [n]} \theta_j,   
      \\
       \sum_{j<k} c_2 (\theta_j+\theta_k)^2 &=  - \dfrac{(n-2)\varLambda}{2} \sum_{j\in [n]}\phi_j^2.
 \end{align*}
 Cancelling the linear term and using the quadratic term  for the gaussian vector with independent components, we get that  $\int_{B} F(\thetavec)d\thetavec$ reduces to estimating 
$
     \E e^{f(\X)},
$
where 
\begin{equation}\label{f-terms} f(\X) = \sum_{\ell\geq 3}\sum_{j<k} c_\ell \Bigl(X_j+X_k- \dfrac{2(1-\beta)}{n} \sum_{j\in[n]} X_j\Bigr)^\ell 
\end{equation}
and $\X$ is a vector with independent components, each $X_i$ with the probability density  proportional to 
$e^{-\frac{(n-2)\varLambda}{2} x^2} \,\One(x\in B)$. 
Applying Theorem \ref{T:decompositions} and taking $m$ to be sufficiently large, we find that  
\begin{equation}\label{cum-terms}
    \E e^{f(\X)} = (1+O(n^{-c})) \exp\left(\sum_{r=1}^m \frac{\kappa_r(f(\X))}{r!}\right) 
\end{equation}
for any given constant $c>0$. Also, for computations of the cumulants,   truncating the gaussian to such a large region does not make much difference, so one can use Isserlis' theorem. 
\end{proof}

In \cite[Section 5]{regular} we computed the first few terms: 
\begin{align*}
  p_1(x) &= \dfrac14 x, \\
  p_2(x) &= -\dfrac14 x^2, \\
  p_3(x) &= \dfrac{1}{24}(2-23x)x^2, \\
  p_4(x) &= \dfrac{1}{24}(22-129x)x^3, \displaybreak[1] \\
  p_5(x) &= -\dfrac{1}{12}(3-115x +483x^2)x^3, \displaybreak[1]\\
  p_6(x) &= -\dfrac{1}{60}(375-6615x+22097x^2)x^4, \\
  p_7(x) &= \dfrac{1}{720}(1046-87318x+1002900x^2-2791541x^3)x^4.
\end{align*}
Numerical experiments suggested that our new approximations work well in the full range of  degrees, so we conjecture the following. 
\begin{conj}\label{Con:bold}
Uniformly for all $1\leq d\leq n-2$ and even $dn$,   
   \[ 
   \RG(n,d) = \Big(\lambda^\lambda (1{-}\lambda)^{1{-}\lambda}\Big)^{\binom n2}
      \binom{n-1}{d}^{\!n}
   \exp\biggl( \,
        \sum_{j=1}^7 \frac{p_j(\varLambda)}{\varLambda^j n^{j-1}} 
          + O\bigg( \frac{1}{d^3 (n-d)^3 n}\bigg) \biggr).
   \] 
\end{conj}

The condition $\min\{d,n-d\}\geq n^{\sigma}$ in Theorem \ref{regularthm} seems to be merely a technical artefact of our proof techniques.
  Indeed, the Taylor expansion \eqref{f-terms} and the cumulant expansion \eqref{cum-terms} converge at rate $n^{\varepsilon}/\sqrt{d}$, so when $d$ is small, more terms are needed to guarantee the desired precision.  However,  for some reason, the leading terms in the cumulant expansion cancel out with the expansion of $\binom{n-1}{d}^n$.   In relation to  Conjecture~\ref{Con:bold}, we expect that $p_8(\varLambda) = O(\varLambda^5)$, but without actually computing this term, one can only rely on the bound $p_8(\varLambda) = O(1)$, which corresponds to the error term in Theorem \ref{regularthm}.
  It would be interesting to find an explanation for these cancellations and justify that this pattern will continue for higher terms. 

  The switching method of McKay and Wormald~\cite{MWsparse} seems incapable
  of reaching the high precision required in Corollary \ref{Con:bold}.
  On the other hand, the method of Liebenau and Wormald \cite{Liebenau} might apply,  provided one can derive a similar, sufficiently accurate formula for graphs with near-regular degree sequences and check that it agrees with the recurrence relation.   

It is quite typical that analytical methods for combinatorial enumeration problems work better in the dense regime (for example, due to the convergence rate issue explained above).  However, it is less known that there is another integral representation of $\RG(n,d)$ that behaves much better for sparse $d$, which we present below.  In particular, de Panafieu \cite{Panafieu} used it to get the series for $\RG(n,d)$ when  $d = O(1)$, which implies Conjecture \ref{Con:bold} for this case.   
Evnin and Horinouchi~\cite{EH2024} noted that by expanding
$\log(1+z_jz_k)$ in~\eqref{reg:generating} and applying the
Hubbard--Stratonovich transformation
\[
   \exp\biggl( \frac{(-1)^{j+1}\bigl(\sum_k z^j\bigr)^2}{2j}\biggr)
   =
   \frac{1}{\sqrt{2\pi j}}\int_{-\infty}^{\infty}
    \exp\biggl( e^{-x^2/(2j)} + \frac{i^{j+1}x\sum_k z_k^j}{j}\biggr)
    \, dx,
\]
we get the formula
\[
  \RG(n,d) = \frac{1}{\sqrt{(2\pi)^dd!}}
  \int_{-\infty}^{\infty}\!\cdots\int_{\infty}^{\infty}
    \exp\biggl( -\sum_{j=1}^d \frac{x_j^2}{2j}
        + n\log P_d(x_1,\ldots,x_d)\biggl)\,
        dx_1\cdots dx_d,
\]
where $P_d(x_1,\ldots,x_n)$ is the polynomial of total degree $d$ defined by
\[
  P_d(x_1,\ldots,x_d) = 
  [z^d] \Biggl((1+z^2)^{-1/2}
     \exp\biggl(\sum_{j=1}^d\frac{i^{j+1}x_jz^j}{j}\biggr)\Biggl).
\]
This integral offers major advantages over the $n$-dimensional integral~\eqref{RGint} for sparse $d$.
Application of Theorem~\ref{complexThm} should enable 
 extending the result of \cite{Panafieu} to the range when $d$ increases no faster than some power of~$n$.

%%%%%%%%%%%%%%%%%%%%%%
%%%%%%%%%%%%%%%%%%%%%%
%%%%%%%%%%%%%%%%%%%%%%
%%%%%%%%%%%%%%%%%%%%%%
%%%%%%%%%%%%%%%%%%%%%%
%%%%%%%%%%%%%%%%%%%%%%

\section{Proof of Theorem \ref{complexThm}}\label{S:decompositions}

In this section, we prove our main result. 
We estimate $\E e^{f(\X)}$ by taking conditional expectations for one variable at a time, similar to the martingale approach of \cite{mother}. Our proof relies on the particular structural property of $f$: it can be decomposed into parts such that the contribution of a part decreases with respect to the number of arguments it depends on. 
To formalise this property, we introduce a Banach algebra of complex decompositions, then later we pass from the decompositions to functions using the Hoeffding decomposition  \cite{Hoeffding1948}.

Nothing in the proof of Theorem~\ref{complexThm} prevents the complex
domain of $f$ from being replaced by a commutative unital Banach algebra.
The case of non-commutative domains such as matrix algebras remains as
an interesting open question.

\subsection{Banach algebra of decompositions}

Consider a map $\pi: 2^{[n]}\times \Om\rightarrow \Complexes$
such that, for any $V \subseteq [n]$,  the value of the function 
$\pi(V,\cdot):\Om \rightarrow \Complexes$ at the point $\x = (x_1,\ldots,x_n) \in \Om$  depends only on those 
components $x_j$ such that $j \in V$. 
Define $ f_\pi:\Om \rightarrow \Complexes$ by
\[
	f_{\pi} = \sum\limits_{V \subseteq [n]} \pi(V,\cdot).
\] 
In other words, the map $\pi$ specifies the way  
$f_\pi$ decomposes into a sum of functions indexed by subsets of $[n]$ with 
indices corresponding to redundancy of arguments. Further, we refer such a map $\pi$ as a {\it decomposition}. 
We start by introducing the algebraic structure of the space of decompositions.

%First, let us introduce the algebraic structure of the space of decompositions. 
For given decompositions $\pi$ and $\omega$ and constants $c_1,c_2 \in \Complexes$
the decompositions $c_1 \pi + c_2 \omega$ and $\pi \omega$ are defined by 
\[
	[c_1\pi+c_2\omega] (V,\cdot) \equiv c_1\pi(V,\cdot) + c_2 \omega(V,\cdot), 
	 \ \ \ \ \ \ 
	[\pi\omega] (V,\cdot) \equiv 
	 \sum_{V_1\cup V_2 =V} \pi(V_1,\cdot)\omega(V_2,\cdot).
\]
Note that
\begin{equation}\label{def:product}
	f_{c_1\pi+c_2\omega} = c_1 f_{\pi} + c_2 f_{\omega} \ \ \text{ and } \ \    
	f_{\pi\omega} =  f_{\pi}  f_{\omega}.
\end{equation}
It is not hard to check that these operations are well-defined and  the usual associative, commutative and distributive properties are inherited. More formally, decompositions form an associative, commutative algebra over $\Complexes$.
It has a multiplicative unit $\One$, given by
\[
	\One (V,\cdot) \equiv 
	\begin{cases}
	\,1,
	& \text{if } V =\emptyset;\\
	\,0, & \text{ otherwise.}
	\end{cases}
\]

We can also define the decomposition corresponding to an infinite series 
$\sum_{j=0}^\infty \pi_j$ by 
\[
	\left[\,\sum_{j=0}^\infty \pi_j\right] (V,\x) = \sum_{j=0}^\infty \pi_j (V,\x),
\]
provided that the right-hand side converges in $\Complexes$ for all $V \subseteq [n]$, $\x \in \Om$,  or, 
in other words, the series $\sum_{j=0}^\infty \pi_j$ converges pointwise.
For two series $\sum_{j=0}^\infty \pi_j$ and 
$\sum_{j=0}^\infty  \omega_j$ consider the Cauchy product 
\[
	\left(\sum_{j=0}^\infty \pi_j\right) \times_c
	\left(\sum_{j=0}^\infty \omega_j\right)= \sum_{j=0}^\infty \rho_j, \ \ \ \ \ \rho_j = \sum_{k=0}^{j} \pi_k \omega_{j-k}.
\]
\begin{lemma}\label{L:series}
	Suppose all three series $\sum_{j=0}^\infty \pi_j$,
	$\sum_{j=0}^\infty \omega_j$ and 
	$\left(\sum_{j=0}^\infty \pi_j\right) \times_c
	\left(\sum_{j=0}^\infty \omega_j\right)$ converge 
	pointwise to decompositions $\pi$, $\omega$, $\tau$, respectively. Then
	$
		\tau = \pi \omega. 
	$
\end{lemma}
\begin{proof}
    We recall that a series of complex numbers $\sum_{j=1}^\infty c_k$ is called 
    Ces\`aro summable with Ces\`aro sum $C$ if 
    \[
    	\lim_{N\rightarrow \infty} \frac{1}{N} \sum_{j=1}^{N} \sum_{k=1}^{j}  c_k = C.
    \]
    It is well known that the Ces\`aro sum agrees with the sum of the series provided the series converges.  
    		We will use  Ces\`aro's theorem: if series of complex numbers $A=\sum_{j=0}^\infty a_j$ and $B=\sum_{j=0}^\infty b_j$
		are convergent then 
	   \[
			\mathit{CS}\left[\left(\sum_{j=0}^\infty a_j\right)\times_c \left(\sum_{j=0}^\infty b_j\right)\right]= AB,
		 \]
		 where $\mathit{CS}[\cdot]$ denotes the Ces\`aro sum.

	 By the assumptions of the lemma,  for any $V_1,V_2 \subseteq [n]$ and $\x \in \Om$, we find that 
	 \[
	 	\mathit{CS}\left[ \left(\sum_{j=1}^\infty \pi_j(V_1,\x)\right) \times_c  \left(\sum_{j=1}^\infty \omega_j(V_2,\x)\right) \right] = \pi(V_1,\x)\omega(V_2,\x).
	 \]
	  Due to convergence of $\left[\left(\sum_{j=0}^\infty \pi_j\right) \times_c
	\left(\sum_{j=0}^\infty \omega_j\right)\right] (V,\x)$, we obtain that
	  \begin{align*}
	  		[\pi \omega] (V,\x)  &= 
	  		\sum_{V_1 \cup V_2 = V}  CS \left[  \left(\sum_{j=1}^\infty \pi_j(V_1,\x)\right) \times_c  \left(\sum_{j=1}^\infty \omega_j(V_2,\x)\right) \right]
	  		\\
	  		&= \mathit{CS} \left[ \sum_{j=1}^\infty \left(\sum_{k=1}^j 
	  		\sum_{V_1 \cup V_2 = V}  \pi_k(V_1,\x) \omega_{j-k}(V_2,\x) \right) \right]
	  		\\&=
	  		  \mathit{CS} \left[ \sum_{j=1}^\infty \left(\sum_{k=1}^j 
	  		 [\pi_k \omega_{j-k}] (V,\x) \right) \right] = \tau (V,\x)
	  \end{align*}	
	   for all $V \subseteq[n]$, $\x \in \Om$.
\end{proof}

Let $\D$ denote the space of bounded decompositions:
\[
	\D = \{\pi \st  \pi \text{ is a decomposition and }   \max_{V\subseteq [n]}\|\pi(V,\cdot)\|_\infty< \infty \},
\]
where  $\norm{f}_\infty =\sup_{\x \in \Om} |f(\x)|$. 
For any $\gamma\geq 0$, let  
\[
	\norm{\pi}_{\gamma} = \sum_{V\subseteq [n]}  \|\pi(V,\cdot)\|_\infty \, e^{ \gamma|V|}.
\]
Since the number of summands $\pi(V,\cdot)$ is finite, we find that  $\norm{\cdot}_{\gamma}$ is a norm on $\D$ as all the norm axioms are inherited from the uniform norm.
  Moreover, $\{\D,\norm{\cdot}_{\gamma} \}$ is complete because  the space
   $\{\phi \st \phi: \Om\rightarrow \Complexes \text{ and } \norm{\phi}_\infty<\infty\}$ equipped with the norm $\norm{\cdot}_\infty$ is a Banach space. 
The norm $\norm{\cdot}_{\gamma}$ is submultiplicative. 
Indeed, for given decompositions $\pi$ and $\omega$, 
\begin{equation}\label{submult_gamma}
\begin{aligned}
	\norm{\pi \omega}_{\gamma} &= 
	\sum_{V\subseteq [n]} \,\sup_{\x \in \Om} \left|\sum_{V_1\cup V_2 =V} \pi(V_1, \x)\omega(V_2, \x)\right|  e^{ \gamma|V|} 
	\\
	&\leq 
	\sum_{V_1,V_2\subseteq [n]}  \sup_{\x \in \Om} \, \abs{\pi(V_1, \x)}
	 e^{ \gamma|V_1|} \, \sup_{\x \in \Om} \,\abs{\omega(V_2, \x)}  e^{ \gamma|V_2|} =  \norm{\pi}_\gamma \norm{\omega}_{\gamma} .
\end{aligned}
\end{equation}
Overall, we conclude  that $\{\D,\norm{\cdot}_\gamma\}$ is a unital commutative  Banach algebra.

For any integer $j\geq 0$,  the  powers of decompositions are defined by $\pi^j = \One \prod_{k=1}^j \pi$.    From \eqref{def:product}, we know that $f_{\pi^j} = (f_{\pi})^j$.
Observing  that $|f_\pi|^j = |f_{\pi^j}| \leq \norm{\pi}_\gamma^j$, we obtain the following lemma. 
\begin{lemma}\label{L:series_gamma}
Let $\pi \in \mathcal D$ and $c_0,c_1\ldots \in \Complexes$ be such  that $\sum_{j=0}^\infty |c_j| \norm{\pi}_\gamma^j < \infty$ for some $\gamma>0$. 
Then, the series $\pi^* = \sum_{j=0}^\infty c_j \pi^j$ converges in $\{\D, \|\cdot\|_\gamma\}$ and 
\[	
\norm{\pi^*}_\gamma \leq \sum_{j=0}^\infty \,\abs{c_j}\, \norm{\pi}_\gamma^j.
\]
Also, the series $\sum_{j=0}^\infty c_j f_\pi^j$ converges to $f_{\pi^*}$ uniformly.
\end{lemma}

 We can define the exponential function of a decomposition by
\[
	e^\pi  = \sum_{j=0}^\infty \dfrac{1}{j!}\, \pi^j.
\]  
 We find from Lemma \ref{L:series_gamma} that  
$\sum_{j=0}^N \frac{1}{j!}\, \pi^j(V,\cdot) $ converges to $e^{\pi}(V,\cdot)$ uniformly   and
\begin{equation}\label{def:exp}
    f_{e^{\pi}} = e^{f_\pi}, \qquad 
\|e^\pi\|_\gamma \leq e^{\norm{\pi}_\gamma}.
\end{equation}
Using Lemma \ref{L:series} as well, we find that for $\pi,\omega \in \D$, 
\[
	e^\pi e^{\omega} = e^{\pi + \omega}.
\]

%%%%%%%%%%%%%%%%%%%%%%%%%%%%%%%%%%%%%%%%%%%%%%%%%%%%%%%%%%%%%%%%%%%%%%%%%%%%%%%%%%%%%%%%%%%
%%%%%%%%%%%%%%%%%%%%%%%%%%%%%%%%%%%%%%%%%%%%%%%%%%%%%%%%%%%%%%%%%%%%%%%%%%%%%%%%%%%%%%%%%%%
%%%%%%%%%%%%%%%%%%%%%%%%%%%%%%%%%%%%%%%%%%%%%%%%%%%%%%%%%%%%%%%%%%%%%%%%%%%%%%%%%%%%%%%%%%%%%%%%%%%%%%%%%%%%%%%%%%%%%%%%%%%%%%%%%%%%%%%%%%%%%%%%%%%%%%%%%%%%%%%%%%%%%%%%%%%%%%%%%%%%%%
%%%%%%%%%%%%%%%%%%%%%%%%%%%%%%%%%%%%%%%%%%%%%%%%%%%%%%%%%%%%%%%%%%%%%%%%%%%%%%%%%%%%%%%%%%%
%%%%%%%%%%%%%%%%%%%%%%%%%%%%%%%%%%%%%%%%%%%%%%%%%%%%%%%%%%%%%%%%%%%%%%%%%%%%%%%%%%%%%%%%%%%

\nicebreak
\subsection{Expectation and difference operators}

Let us fix some distribution on $\Om = \varOmega_1\times\cdots\times \varOmega_n$ given by an
$\Om$-valued random vector $\X=(X_1,\ldots,X_n)$  with independent components. 
In the following, we restrict ourselves to
 {\it measurable} decompositions  $\pi$, which means, for all $V \subseteq[n]$,  that the functions $\pi(V,\X)$ are measurable with respect to the sigma field $\sigma(\X)$. 
 %We denote by $L_\gamma(\D)$ the space of measurable decompositions of $\D$ equipped with the norm $\norm{\cdot}_\gamma$.
 Note that $\pi \in \D$ implies that the random variables $\pi(V,\X)$ are bounded for all $V\subseteq[n]$, 
 which is sufficient for the existence of moments.

For each $j \in [n]$, define the expectation operator  $\E^{j}: \D \rightarrow \D$   by
\[	
	\E^{j} \pi (V,\x) = \begin{cases} 
	\,\pi(V,\x) + \E\(\pi({V\cup \{j\}}, x_1, \ldots, x_{j-1}, X_j, x_{j+1}, \ldots, x_n)\),
	& \text{if } j\notin V;\\
	\,0, & \text{ otherwise.}
	\end{cases}
\]
It satisfies the  linearity property:
\[	
	\E^{j} [c_1 \pi + c_2 \omega] = c_1 \E^{j} \pi + c_2 \E^{j} \omega
\]
for any $c_1,c_2 \in \Complexes$ and  $\pi,\omega \in \D$. 
Note also that 
\[
\E^{j} \E^{j} \pi = \E^{j} \pi  \
\text{ and }  \ \E^{j} \E^{k} \pi = \E^{k}\E^{j} \pi
\]  
for any $j,k \in [n]$.  In addition, we have 
\begin{equation}\label{expect-func}
	\E( f_\pi(\X) \mid X_1,\ldots,X_{j-1}, X_{j+1}, \ldots, X_n) = f_{\E^{j}\! \pi} (\X).
\end{equation}

The decomposition $\pi$  is called {\it $j$-restricted} if 
$\pi(V,\cdot) \equiv 0$ for any $V\subseteq[n]$ such that $j\notin V$.  Similarly, we call 
 it {\it $j$-free} if 
$\pi(V,\cdot) \equiv 0$ for any $V\subseteq[n]$ such that $j\in V$. 
Observe that 
\begin{equation}\label{split}
\begin{aligned}
\text{
for any given decomposition $\pi$ there is  unique way to split it into} \\
\text{the sum of a  $j$-restricted and a $j$-free decomposition.}
\end{aligned}
\end{equation} 
Indeed, any non-trivial $\pi(V,\cdot)$, $V \subseteq [n]$, is present exactly in one of them, corresponding to the cases
$j \in V$ and $j \notin V$, respectively.
\begin{lemma}\label{dep/indep}
  Let $\pi$, $\omega$ be decompositions from $\D$. For $j \in[n]$, the following hold. 
  \begin{itemize}\itemsep=0pt
			     
   	\item[(a)] If  $\pi$ is $j$-restricted, then $\pi\omega$ is also $j$-restricted. 
   	
   	\item[(b)] If   $\pi$ is  $j$-restricted and $\omega$ is  $j$-free,
  	 then $\norm{\pi + \omega}_\gamma = \norm{\pi}_\gamma+ \norm{\omega}_\gamma$. 
  	
  	 \item[(c)] 
  	  $\E^{j} \pi$ is a  $j$-free decomposition with $\norm{\E^{j} \pi}_\gamma \leq \norm{\pi}_\gamma$. 
In addition, if $\pi$ is  $j$-restricted, then	
$\norm{\E^{j} \pi }_\gamma \leq e^{-\gamma}\norm{\pi}_\gamma.$
	\item[(d)] If   $\pi$ is $j$-free, then	
	        $\E^{j} [\pi \omega] = \pi \E^{j} \omega$.
  \end{itemize}
\end{lemma}
\begin{proof}
	Claim (a) follows directly from the definition of the product of decompositions. 
Claims (b,d) and the first part of claim (c) are straightforward from the definitions and the independence of components of $\X$. 
	The second part of claim~(c) is  by observing
    \[
    \sup_{\x \in S} |\E^{j}\pi(V,\x)| \leq \sup_{\x \in S} |\pi(V \cup\{j\},\x)|
    \]
	for all non-trivial functions $\E^{j}\pi(V,\cdot)$  (i.e. when $j \notin V$), 
	but $|V| = |V \cup\{j\}| -1$, which gives the additional factor of $e^{-\gamma}$ in the norm $\|\cdot\|_\gamma$. 
\end{proof}

Next, for $\yvec \in \Om$,
we consider the operator $R^j_{\yvec}:\D\rightarrow D $ defined by 
\[
    R^j_{\yvec} \pi (V,\x) =   
	\begin{cases}
	\pi(V,\x) + \pi({V\cup \{j\}}, x_1, \ldots, x_{j-1}, y_j, x_{j+1}, \ldots, x_n)),
	& \text{if } j\notin V,\\
	0, & \text{ otherwise.}
	\end{cases} 
\]
Similarly to $\E^{j} \pi$ defined in the previous section, we have that 
\[
\text{$R^j_{\yvec} \pi$ is a  $j$-free decomposition and }
\|R^j_{\yvec} \pi\|_\gamma \leq \|\pi\|_{\gamma}.
\]
For  $V = \{v_1,\ldots, v_k\}\subseteq [n]$, define 
\[
R^{V}_{\yvec}= R^{v_1}_{\yvec}\cdots R_{\yvec}^{v_k}
\qquad 
\mbox{ and }
\qquad
\partial^V_{\yvec}:= \partial_{\yvec}^{v_1}\cdots \partial_{\yvec}^{v_k},
\]
where  $\partial_{\yvec}^j:= I - R^j_{\yvec}$ and  $I$ is the identity operator.  
  These definitions do not depend on the order of the elements of $V$ since the operators $ R^j_{\yvec} $  and $ R^{j'}_{\yvec} $ commute for any $j,j'\in [n]$. 
We set  $R^{\emptyset}_{\yvec} 
= \partial^{\emptyset}_{\yvec} = I$.
We will use the identity
\[
 \partial_{\yvec}^V  = \sum_{W\subseteq V} (-1)^{|W|} R_{\yvec}^W.
\]
Let
\[
	 \Gamma_{V,\gamma}(\pi)   :=  \sup_{ \y\in \Om} \left\|\partial^V_{\yvec}[\pi] \right\|_\gamma.
\]
In particular,
$\Gamma_{\emptyset,\gamma}(\pi) = \|\pi\|_\gamma.$

%%%%%%%%%%%%%%%%%%%%%%%%%%%%%%%%%%%%%%%%%%%%%%%%%%%%%%%%%%%%%%%%%%%%%%%%%%%%%%%%%%%%%%%%%%%
%%%%%%%%%%%%%%%%%%%%%%%%%%%%%%%%%%%%%%%%%%%%%%%%%%%%%%%%%%%%%%%%%%%%%%%%%%%%%%%%%%%%%%%%%%%
%%%%%%%%%%%%%%%%%%%%%%%%%%%%%%%%%%%%%%%%%%%%%%%%%%%%%%%%%%%%%%%%%%%%%%%%%%%%%%%%%%%%%%%%%%%%%%%%%%%%%%%%%%%%%%%%%%%%%%%%%%%%%%%%%%%%%%%%%%%%%%%%%%%%%%%%%%%%%%%%%%%%%%%%%%%%%%%%%%%%%%
%%%%%%%%%%%%%%%%%%%%%%%%%%%%%%%%%%%%%%%%%%%%%%%%%%%%%%%%%%%%%%%%%%%%%%%%%%%%%%%%%%%%%%%%%%%
%%%%%%%%%%%%%%%%%%%%%%%%%%%%%%%%%%%

\subsection{Cumulants of decompositions}

For each $j\in [n]$, 
 consider the operator $\E^{\geq j}: \D\rightarrow \D$ defined by
\[
\E^{\geq j}  =   \E^{j}\E^{j+1} \ldots \E^{n}.
\]
 For  $\pi_1,\ldots,\pi_r\in \D$,  consider the conditional joint cumulant  defined by
 \[
 	\kappa^{\geq j}[\pi_1, \ldots, \pi_r] = \sum_{\tau \in \mathcal P_r} \, (\card{\tau}-1)!\, (-1)^{\card{\tau}-1}
 	\prod_{B\in\tau} \E^{\geq j}\Bigl[\prod_{k \in B} \pi_k\Bigr],
 \]
 where  $\mathcal P_r$ denotes the set of   unordered partitions $\tau$ of~$[r]$ (with non-empty blocks)  and $|\tau|$ denotes  the number of blocks in the partition $\tau$. 
 We also set
 \begin{equation}\label{cum-n+1}
 \E^{\geq n+1} \pi   = \kappa^{(\geq n+1)}\pi:= \pi \ \ \  \text{ and } \ \
\kappa^{\geq n+1}[\pi_1,\ldots,\pi_r] :=  0 
\text{ for } r \geq 2.
\end{equation}
By definition, the conditional cumulant 
$\kappa^{\geq j}$ is a symmetric and multilinear $\D$-valued function of its arguments.  
 The conditional cumulant of order $r$ is defined by
 \[
 \kappa_r^{\geq j}[\pi] =  \kappa^{\geq j}[\underbrace{\pi,\ldots, \pi}_{r \text{ times}}].
 \]
Using Lemma \ref{dep/indep}(c) and \eqref{expect-func},  we find that  decompositions
$ \E^{\geq 1} \pi$  and $ \kappa^{\geq 1}  \pi$
are $j$-free for all $j\in[n]$ and 
\begin{equation}\label{cum-functions}
    \E^{\geq 1}  \pi = \One \cdot \E f_{\pi}(\X), \qquad 
    \kappa_r^{\geq 1} [\pi] = \One  \cdot \kappa_r (f_{\pi}(\X)),
 \end{equation}
where $\kappa_r (f_{\pi}(\X))$ are the usual cumulants of the random variable $f_{\pi}(\X)$ defined according to \eqref{cum:def}. We will need the following bounds.
\begin{lemma}\label{L:cum-crucial}
Let  $m$ be a positive integer and $\alpha, \gamma \geq 0$.
 Assume that $\pi \in \D$ is such that 
    \[
	 \max_{\substack{t \in [m]\\ j \in [n]}}\,\sum_{V\in \binom{[n]}{t} \st j \in V } \!\Gamma_{V,\gamma}(\pi) \leq \alpha.
 \]
 Then the following bounds hold:
\begin{itemize}
  \item[(a)]
  for any $r\in [m]$ and  $j\in [n]$, 
 \[ \left\|\kappa_{r}^{\geq j +1 } \pi  -\kappa_{r}^{\geq j} \pi\right\|_\gamma \leq   1.1\cdot 80^{r-1} \frac{(r-1)!}{r} \alpha^r;\]
 \item[(b)] for any  $j\in [n]$,
 \[
 		\left\|\E^{\geq j}\exp\left(\sum_{r=1}^m 
 		\frac{\kappa_{r}^{\geq j +1 } [f]  -\kappa_{r}^{\geq j }[\pi]} {r!}
 		\right)-1\right\|_{\textstyle \gamma} \leq e^{ (100\alpha)^{m+1} }-1.
 		\]
 \end{itemize}
\end{lemma}
The proof of Lemma \ref{L:cum-crucial} is identical to \cite[Lemma 6.6]{eulerian} and \cite[Lemma 6.7]{eulerian}. We follow the arguments of \cite[Section 6]{eulerian} applied to decompositions instead of functions. All estimates and identities  remain the same except for the minor notational changes:  we use 
norm $\|\cdot\|_\gamma$  instead of the  norm $\|\cdot\|_\infty$ and $\Gamma_{V,\gamma}$ plays the role of $\Delta_V$ in \cite[Section 6]{eulerian}.

\begin{thm}\label{T:decompositions}
	Let $m$ be a positive integer, $\alpha \in [0,\frac{1}{100}]$ and $\gamma \geq \frac32 \alpha$. 
	 Assume that $\pi \in \D$ is such that
\[
	 \max_{\substack{t \in [m]\\ j \in [n]}}\,\sum_{V\in \binom{[n]}{t} \st j \in V } \!\Gamma_{V,\gamma}(\pi) \leq \alpha.
\]
	Then, we have
   \[
  	 \E { e^{f_\pi(\X)} }  =  
   (1+\delta)^n  \exp\left(\sum_{r=1}^m \frac{\kappa_r (f_\pi(\X))}{r!} \right),
  \]
   where $\delta \in \Complexes$ satisfies $\abs\delta \leq  e^{(100 \alpha)^{m+1}}-1$. 
   Furthermore,
   for any $r\in \{2,\ldots,m\}$,
   \[
            | \kappa_r (f_\pi(\X))| 
            \leq  n \dfrac{(r-1)!}{50r}(80
             \alpha)^{r}.  
   \]
\end{thm}

\begin{proof}

For $j \in [n{+}1]$, define the decomposition $\lambda_j$ by 
\begin{equation*}
	\lambda_j = \E^{\geq j}  \exp\left(\pi - \sum_{r=1}^m \frac{\kappa_{r}^{\geq j} [\pi]}{r!} \right)  -\One.
\end{equation*}
We prove by downwards induction that 
 $\|\lambda_j\|_\gamma \leq  e^{(n+1-j)(100 \alpha)^{m+1}}-1$.
From \eqref{cum-n+1} we  have that
\[
      \pi = \sum_{r=1}^m \frac{\kappa_{r}^{\geq n+1} [\pi]}{r!}. 
\]
Thus, $\|\lambda_{n+1}\|_\gamma =0$ so the base of the induction holds.

We proceed to the induction step.
For  $j\in [n]$,  we can write 
\begin{equation}\label{rec:lj}
 \begin{aligned}
	\lambda_{j} &= 
    \E^{\geq j} \E^{\geq j+1} \exp\left(\pi - \sum_{r=1}^m \frac{\kappa_{r}^{\geq j} [\pi]}{r!} \right)  -\One.
    \\&=
	\E^{(\geq j)} \left[ (\One+\lambda_{j+1}) \exp\left( \sum_{r=1}^m \frac{\kappa_{r}^{(\geq j+1)}(\pi) - \kappa_{r}^{(\geq j)}(\pi)}{r!} 
	\right)\right] - \One\\
 &= \E^{\geq j} \left[  (\One+\lambda_{j+1}') 
 \left(\exp\left( \sum_{r=1}^m \frac{\kappa_{r}^{\geq j+1}[\pi] - \kappa_{r}^{\geq j}[\pi]}{r!} \right)-\One\right)\right] + \E^{\geq j}  \lambda_{j+1}',\\
	&\qquad \qquad \qquad +  \E^{\geq j}  (\lambda_{j+1}- \lambda_{j+1}') \exp\left( \sum_{r=1}^m \frac{\kappa_{r}^{\geq j+1}[\pi] - \kappa_{r}^{\geq j}[\pi]}{r!} \right).
	\end{aligned}
\end{equation}
where $\lambda_{j+1}'$ is $j$-free and
$\lambda_{j+1}- \lambda_{j+1}'$ is $j$-restricted;
such a split is unique, see~\eqref{split}.
It follows from Lemma \ref{dep/indep}(c) that $\E^{\geq j} e^{\pi} $ and $\kappa_{r}^{\geq j}[\pi]$ are $k$-free decompositions for any $k \in \{j+1,\ldots,n\}$. Therefore, $\lambda_{j+1}'$ is $k$-free for any $k \in \{j,\ldots,n\}$.
Thus, using \eqref{submult_gamma}, Lemma~\ref{dep/indep}(c,d), and Lemma \ref{L:cum-crucial}(b), we find that 
\begin{align*}
	&\left\|\E^{(\geq j)}(\One+\lambda_{j+1}') \left(\exp\left(
	 \sum_{r=1}^m \frac{\kappa_{r}^{\geq j+1}[\pi] - \kappa_{r}^{\geq j}[\pi]}{r!} \right)-\One\right)\right\|_\gamma 
     \\
    &=\left\|(\One+\lambda_{j+1}') \E^{\geq j} \left[\exp\left(
	 \sum_{r=1}^m \frac{\kappa_{r}^{\geq j+1}[\pi] - \kappa_{r}^{\geq j}[\pi]}{r!} \right)-\One\right]\right\|_\gamma 
	\\ &\leq \norm{\One+\lambda_{j+1}'}_\gamma (e^{ (100\alpha)^{m+1} }-1).
\end{align*}
From Lemma \ref{dep/indep}(c),
we get
$\norm{\E^{\geq j}  \lambda_{j+1}'}_\gamma \le \norm{\lambda_{j+1}'}_\gamma$.
 Using \eqref{submult_gamma} again, Lemma \ref{dep/indep}(a,c), and Lemma \ref{L:cum-crucial}(a), we obtain that 
 \begin{align*}
 	 &\left\| \E^{\geq j} \left[(\lambda_{j+1}- \lambda_{j+1}') \exp\left( \sum_{r=1}^m \frac{\kappa_{r}^{\geq j+1}[\pi] - \kappa_{r}^{\geq j}[\pi]}{r!} \right) \right]\right\|_\gamma\\  &\leq 
 	  e^{-\gamma} \left\|(\lambda_{j+1}- \lambda_{j+1}')\exp\left( \sum_{r=1}^m \frac{\kappa_{r}^{\geq j+1}[\pi] - \kappa_{r}^{\geq j}[\pi]}{r!} \right) \right\|_\gamma 
      \\  &\leq 
 	  e^{-\gamma} \left\|(\lambda_{j+1}- \lambda_{j+1}')\right\|_\gamma
      \exp\left( \sum_{r=1}^m \frac{\|\kappa_{r}^{\geq j+1}[\pi] - \kappa_{r}^{\geq j}[\pi]\|_\gamma}{r!} \right)  
      \\
 	 &\leq \exp\biggl(-\gamma + 1.1 \alpha  \sum_{r=1}^m \frac{(80\alpha)^{r-1}}{r^2}\biggr) 
	 \norm{\lambda_{j+1}- \lambda_{j+1}'}_\gamma \\
     &\leq
 	 \norm{\lambda_{j+1}- \lambda_{j+1}'}_\gamma,
 \end{align*}
 where the last inequality relies on the assumptions $\alpha \leq \frac{1}{100}$ and $\gamma \geq \frac32 \alpha$, which ensure that
 \[
    1.1\, \alpha  \sum_{r=1}^m \frac{(80\alpha)^{r-1}}{r^2} \leq
    1.1\, \alpha  \sum_{r=1}^\infty \frac{0.8^{r-1}}{r^2}\leq 
    \tfrac32 \,\alpha \leq \gamma.
 \]
  Next, due to Lemma~\ref{dep/indep}(b), we have
 \[
 \norm{\lambda_{j+1}}_\gamma = \norm{\lambda_{j+1}'}_\gamma 
 + \norm{\lambda_{j+1}-\lambda_{j+1}'}_\gamma.
 \]
 Combining this all together,  
  we derive from \eqref{rec:lj} that
 \begin{align*}
 	\norm{\lambda_j}_\gamma &\leq \norm{\lambda_{j+1}}_{\gamma} + (1+ \norm{\lambda_{j+1}}_{\gamma}') (e^{ (100\alpha)^{m+1}}-1)
    \\ &\leq \norm{\lambda_{j+1}}_{\gamma}
    \,e^{ (100\alpha)^{m+1}} + e^{ (100\alpha)^{m+1}}-1.
 \end{align*}
Now the induction step is straightforward from the induction hypothesis, completing the proof of the claimed bound for $\|\lambda_{j}\|_\gamma$.

 Using \eqref{def:exp} and \eqref{cum-functions}, we find that 
\[
	\E^{\geq 1} e^{\pi} =  \E  e^{f_\pi(\X)}\cdot\One, \ \ \ \ \kappa_{r}^{\geq 1} [\pi] =  \kappa_r (f_\pi(\X))\cdot \One.
\]
Then, since $\|\lambda_1\|_\gamma \leq e^{n(100 \alpha)^{m+1}}-1$, we get that 
\[
\left| \E { \exp \left( f_\pi(\X)  - \sum_{r=1}^m \frac{\kappa_r (f_\pi(\X))}{r!}\right)}  -1  \right|
   = \|\lambda_1\|_\gamma \leq e^{n(100 \alpha)^{m+1}}-1.
\]
This  implies that 
\[
    \E { \exp \left( f_\pi(\X)  - \sum_{r=1}^m \frac{\kappa_r (f_\pi(\X))}{r!}\right)} = (1+\delta)^n
\]
for some $\delta \in \Complexes$ with $|\delta| \leq e^{(100 \alpha)^{m+1}}-1$, as required.

    By Lemma \ref{L:cum-crucial}(a),
we have that,
for any $r \in \{2,\ldots,m\}$ and  $j\in [n]$,
\[
\left\|\kappa_{r}^{\geq j +1 } [ \pi ]  
- \kappa_{r}^{\geq j}[  \pi ] \right\|_\gamma
\leq   1.1\cdot 80^{r-1} \dfrac{(r-1)!}{r} \alpha^r
\leq 
\dfrac{(r-1)!}{50r} (80\alpha)^r.
\]
Recalling from \eqref{cum-n+1} that $\kappa_{r}^{\geq n+1}[ \pi]=0$ and using the triangle inequality, we obtain
\[
| \kappa_{r}\left( f_\pi(\X) \right) |
 = \| \kappa_{r}^{\geq 1}[ \pi]  \|_\gamma
\le 
\sum_{j=1}^{n}\, \left\|\kappa_{r}^{\geq j +1 } [  \pi ]  
- \kappa_{r}^{\geq j}[  \pi ] \right\|_\gamma \leq  n\dfrac{(r-1)!}{50r} (80\alpha)^r.
\]
 This completes the proof.
\end{proof}

%%%%%%%%%%%%%%%%%%%%%%%%%%%%%%%%%%%%%%%%%%%%%%%%%%%%%%%%%%%%%%%%%%%%%%%%%%%%%%%%%%%%%%%%%%%%%%%%
%%%%%%%%%%%%%%%%%%%%%%%%%%%%%%%%%%%%%%%%%%%%%%%%%%%%%%%%%%%%%%%%%%%%%%%%%%%%%%%%%%%%%%%%%%%%%%%%
%%%%%%%%%%%%%%%%%%%%%%%%%%%%%%%%%%%%%%%%%%%%%%%%%%%%%%%%%%%%%%%%%%%%%%%%%%%%%%%%%%%%%%%%%%%%%%%%
%%%%%%%%%%%%%%%%%%%%%%%%%%%%%%%%%%%%%%%%%%%%%%%%%%%%%%%%%%%%%%%%%%%%%%%%%%%%%%%%%%%%%%%%%%%%%%%%
%%%%%%%%%%%%%%%%%%%%%%%%%%%%%%%%%%%%%%%%%%%%%%%%%%%%%%%%%%%%%%%%%%%%%%%%%%%%%%%%%%%%%%%%%%%%%%%%

%%%%%%%%%%%%%%%%%%%%%%%%%%%%%%%%%%%%%%%%%%%%%%%%%%%%%%%%%%%%%%%%%%%%%%%%%%%%%%%%%%%%%%%%%%%%%%%%%%%%%%%%%%%%%%
\subsection{From decompositions to functions}

Recall that  $\X=(X_1,\ldots,X_n)$ be a random vector with independent components
 taking values in $\Om =\varOmega_1\times\cdots\times \varOmega_n$.  
For a measurable bounded $f:\Om \rightarrow \Complexes$, we consider the  decomposition $\pi_f$  introduced  by  Hoeffding \cite{Hoeffding1948} defined by  
 \[
 		\pi_{f} (V,\x) = \sum_{W\subseteq V} (-1)^{|V|+|W|} \E (f(\X)\mid \X_W = \x_W),
 \]
 where vectors $\X_W$, $\x_W$ consist of the components of $\X$, $\x$ corresponding to elements of~$W$. 
 The Hoeffding decomposition  is linear: for any $c_1,c_2 \in \Complexes$, and measurable bounded functions $f,g: \Om \rightarrow \Complexes$, 
 \[
    \pi_{c_1 f  + c_2 g} = c_1 \pi_f + c_2 \pi_g.
 \]
 Observe also that 
 \begin{equation}\label{H-observ}
 \begin{aligned}
 \text{if  $f(\X)$ is independent of   $X_j$ and $j\in V$ then $\pi_{f} (V,\x) = 0$. }  
 \end{aligned}
 \end{equation}
Indeed, according to the definition above, the term 
 $\E (f(\X)\mid \X_W = \x_W)$ would cancel out $\E (f(\X)\mid \X_{W\cup j} = \x_{W\cup j})$ in this case.

In this section, we derive Theorem \ref{complexThm} by  applying  Theorem \ref{T:decompositions} to  the Hoeffding decomposition $\pi_f$.  To do this, we first need to bound the quantities
\[ 
\Gamma_{V,\gamma}(\pi_f)   =  \sup_{ \y\in \Om} \left\|\partial^V_{\yvec}[\pi_f] \right\|_\gamma.\]
Recall from \eqref{def:newDelta} that 
\[
\newDelta_V(f) = \sup_{\yvec\in\Om}\;
   \biggl| \,\E\biggl(\sum_{W\subseteq V} (-1)^{\card{W}} f(\X\fix_W\yvec)
   \biggr)\biggr|,
\]
where $\xvec\fix_V\yvec = (u_1,\ldots,u_n)$
where, $u_j=x_j$ if $j\notin V$ and $u_j=y_j$ if $j\in V$.
\begin{lemma}\label{L:partial-Delta}
    For any $\xvec,\yvec\in \Om$ and $V,W \subseteq [n]$, we have
        $|\partial_y^V \pi_f(W,x)| \leq \newDelta_{V\cup W}(f)$.
\end{lemma}
\begin{proof}
First, we observe that $R_{\yvec}^j \pi_f$ is the Hoeffding decomposition of
the function $h_j(x) = f(\xvec\fix_j\yvec)$.
Indeed, by definition, we have   
\[
   R_{\yvec}^j \pi_f (U,\xvec)
    =  
    \begin{cases}
	\pi_f(U,\xvec)+ \pi_f({U\cup \{j\}},  \xvec\fix_j\yvec ),
	& \text{if } j\notin U,\\
	0, & \text{otherwise.}
	\end{cases} 
 \]
 Using \eqref{H-observ}, we get $\pi_{g_j}(U,\xvec) =  R_{\yvec}^j \pi_f (U,\xvec) = 0$ if $j\in U$. For $j \notin U$,  we get
 \begin{align*}
    \pi_f(U,\xvec) &+ \pi_f({U\cup \{j\}},  \xvec\fix_j\yvec) 
    \\ &=\pi_f(U,\xvec) + \sum_{W'\subseteq U\cup \{j\}} (-1)^{|U|+1+|W'|} \E (f(\X)\mid \X_{W'} = (\x\fix_j \yvec)_{W'}),
    \\
    &=   
    \sum_{W'\subseteq U\cup \{j\} \st j \in W'} (-1)^{|U|+1+|W'|} \E (f(\X)\mid \X_{W'} = (\x\fix_j \yvec)_{W'})
    \\ 
    &=    \sum_{W''\subseteq V} \E (f(\X
    \fix_j \yvec)\mid \X_{W''} = x_{W''})
     = \pi_{h_j}(U,\xvec).
 \end{align*}
 By linearity, we also get  that $\partial_{\yvec}^{j} \pi$ is the Hoeffding decomposition of $f(\xvec)- f(\xvec\fix_j\yvec)$. 
 Applying this for several $j$, we find that
 \[
    \partial_{\yvec}^{V} \pi_f  = \pi_{g^V}, \qquad \text{where }\; g^V(\xvec) = \sum_{U\subseteq V} (-1)^{|V|+|U|} f(\xvec\fix_U \yvec) .
 \]
% Note that $g_V(\xvec)$  is independent of all components $y_j$ with $j\notin V$.
 
 Now we are ready to prove the stated bound. First, we consider the case when $V$ and $W$ are disjoint. Let $\xvec \in \Om$. 
 Note that for any $U \subseteq W$ and $U'\subseteq V$, we have 
 \[
    \E \left(f(\X \fix_{U'}\yvec) \mid \X_U =\xvec_U\right) = \E f(\X \fix_{U\cup U'} \hat{\xvec}),
 \]
 where    $\hat{\xvec} =  \xvec \fix_{V}\yvec$.
 Then, we get
  % Consider any $\yvec'\in \Om$ and let $\hat{\yvec} = \yvec \fix_W \yvec'$. 
  % Observing that $g_V(\xvec \fix_U\yvec) = 
  % g_V(\xvec \fix_U \yvec')$, we get
 \begin{align*}
    \left|\partial_{\yvec}^V \pi_{f}(W,\xvec)\right| &=  \left| \pi_{g^V}(W,\xvec)\right|
    \\ &= \left|\sum_{U \subseteq W} (-1)^{|W|} \E \left(g^V(\X)   \mid \X_U =\xvec_U\right)\right|
    \\&= 
    \left|\sum_{U \subseteq W} 
    \sum_{U' \subseteq V}
    (-1)^{|U|+|U'|} \E \left(f(\X \fix_{U'}\yvec) \mid \X_U =\xvec_U\right)\right|\\
    & = \left|\E \left(\sum_{U'' \subseteq W\cup V} (-1)^{|U|}  f(\X \fix_{U''} \hat{\xvec})\right)\right| \leq \newDelta_{W\cup V}(f).
 \end{align*}
 
 Finally, for the case when $V\cap W \neq \emptyset$, we observe that  
 \[
    \partial^j_{\yvec} \pi_{g}(W,\xvec) = \pi_g (W,\xvec) \qquad \text{for all $j\in W$ and $\xvec\in \Om$.} 
 \]
It follows that 
\begin{align*}
    \partial^{V}_{\yvec} \pi_f (W,\xvec)
    &= \left[\partial^{V\cap W}_{\yvec}  \partial^{V\setminus W}_{\yvec} \pi_f\right](W,\xvec)  = 
     \partial^{V\cap W}_{\yvec} \pi_{g^{V\setminus W}}(W,\xvec) 
     \\ &= \pi_{g^{V\setminus W}}(W,\xvec) = 
     \partial^{V\setminus W}_{\yvec} \pi_f(W,\xvec),
\end{align*}
which reduces the problem to the case when $V$ and $W$ are disjoint.
\end{proof}

Now we can complete the proof of Theorem~\ref{complexThm}.
From Lemma \ref{L:partial-Delta} we derive that 
\[
    \Gamma_{V,\gamma}(\pi_f) \leq \sum_{W \subseteq [n]}  \newDelta_{W\cup V}(f) e^{\gamma |W|} \leq \sum_{W'\in [n]\st V\subseteq W' }  \newDelta_{W'}(f) e^{\gamma |W'|} 2^{|V|}.
\]
Then, recalling from \eqref{Sdef} that 
\begin{equation*}
   S_v = \max_{j \in [n]}\,\sum_{V\in \binom{[n]}{v} \st j \in V } \!\newDelta_V(f), 
\end{equation*}
we find that, for any $t\in [m]$ and $j\in [n]$ 
\[
    \sum_{V\in \binom{[n]}{t} \st j \in V } \!\Gamma_{V,\gamma}(\pi_f)
    \leq 
    \sum_{V\in \binom{[n]}{t} \st j \in V } 
     \sum_{W'\in [n]\st V\subseteq W' }  \newDelta_{W'}(f) e^{\gamma |W'|} 2^{t}
     \leq \sum_{k = t}^n e^{\gamma k} 2^t \binom{k-1}{t-1} S_k(f).
\]
Thus, taking $\gamma = 3\alpha/2$, we have verified that the assumptions of Theorem \ref{T:decompositions} follow from the assumptions of Theorem \ref{complexThm}. Applying  Theorem \ref{T:decompositions}, we get the claimed bounds.

\section{Acknowledgments} 
This work is supported   by the Australian Research Council grant DP250101611.

The author would like to thank Brendan McKay for valuable discussions and feedback on all aspects of this work, and for his patience and understanding, as the extended preparation of this manuscript caused the joint work \cite{regular} to be submitted several years after it was written and presented at multiple conferences and workshops. 

The author thanks Jane Gao for drawing attention to the local limit theorem for triangle counts, which provides a natural application of the main result of this paper, and Oleg Evnin for emphasising the role of the Hubbard–Stratonovich transformation in counting regular graphs.

\end{document}